\documentclass{ifacconf}%
\usepackage{amsmath}
\usepackage{graphicx}
\usepackage{natbib}
\usepackage{amsfonts}
\usepackage{amssymb}%
\providecommand{\U}[1]{\protect\rule{.1in}{.1in}}
\newtheorem{theorem}{Theorem}
\newtheorem{lemma}{Lemma}

\newenvironment{proof}[1][Proof]{\textbf{#1.}}{\rule{0.5em}{0.5em}}
\begin{document}
\begin{frontmatter}
\title{Infinitesimal Perturbation Analysis for Personalized Cancer Therapy Design\thanksref{footnoteinfo}}
\thanks[footnoteinfo]{The authors' work is
supported in part by NSF under Grants CNS-1239021 and IIP-1430145, by AFOSR
under grant FA9550-12-1-0113, and by ONR under grant N00014-09-1-1051.}
\author[First]{Julia L. Fleck}
\author[First]{Christos G. Cassandras}
\address[First]{Division of Systems Engineering and Center for Information and Systems Engineering,
Boston University, Brookline, MA 02446 USA (e-mail: jfleck@bu.edu, cgc@bu.edu)}
\begin{abstract}                
We use a Stochastic Hybrid Automaton (SHA) model of prostate cancer
evolution under intermittent androgen suppression (IAS) to study a threshold-based policy for therapy design. IAS  is currently one of the most widely
used treatments for advanced prostate cancer. Patients undergoing IAS are
submitted to cycles of treatment (in the form of androgen deprivation) and
off-treatment periods in an alternating manner. One of the main challenges in
IAS is to optimally design a therapy scheme, i.e., to determine when to
discontinue and recommence androgen suppression. The level of prostate
specific antigen (PSA) in a patient's serum is frequently monitored to
determine when the patient will be taken off therapy and when therapy will
resume. The threshold-based policy we propose
is parameterized by lower and upper PSA threshold values and is
associated with a cost metric that combines clinically relevant measures of
therapy success. Using Infinitesimal Perturbation Analysis (IPA), we derive
unbiased gradient estimators of this cost metric with respect to the
controllable PSA threshold values based on actual data and show how these
estimators can be used to adaptively adjust controllable parameters so as to
improve therapy outcomes based on the cost metric defined.
\end{abstract}
\end{frontmatter}



\section{INTRODUCTION}

Cancer is widely viewed as a \textquotedblleft disease of
stages\textquotedblright\ \ in which tumors must progress through a series of
discrete \textquotedblleft states\textquotedblright\ in order to ultimately
become malignant [\cite{Hanahan2011}]. This view is particularly suitable to
the development of prostate cancer, which is known to be a multistep process
[\cite{Harrison2012}]. For instance, a patient diagnosed with localized
prostate cancer who has had all the tumor surgically removed is considered to
remain in the state of \textquotedblleft localized disease\textquotedblright%
\ until he progresses to a new state. At each state, distinct therapies can be
prescribed, and the time spent by the patient in any given state is a measure
of the efficacy of the corresponding intervention.

The standard treatment for advanced prostate cancer patients is hormone
therapy in the form of androgen deprivation [\cite{Harrison2012}]. The initial
response to androgen deprivation therapy (ADT) is frequently positive, leading
to a significant decrease in tumor size. However, most patients eventually
develop resistance to ADT and relapse. A generally acceptable mechanism for
explaining such relapse is the existence of an androgen-independent cancer
cell phenotype that is resistant to secondary endocrine therapy and whose
outgrowth leads to tumor recurrence [\cite{Jackson2004}].

Intermittent androgen suppression (IAS) therapy has been recently proposed as
a strategy for delaying or even preventing time to relapse. The purpose of IAS
is to prevent the exisiting tumor from progressing into an
androgen-independent state. In spite of significant clinical experience with
this approach, defining ideal protocols for any given patient remains one of
the main challenges associated with effective IAS therapy [\cite{Hirata2010}].
In fact, recent clinical trials suggest that the success of IAS ultimately
translates into the ability to tailor on and off-treatment schemes to
individual patients [\cite{Bruchovsky2006}, \cite{Bruchovsky2007}]. The design
of optimal personalized IAS treatment schemes remains, therefore, an unsolved problem.

Recent attempts at addressing this problem have led to the development of
several mathematical models that explain the evolution of prostate cancer
under hormone therapy. The model from \cite{Jackson2004} describes the growth
of prostate tumors formed by two subpopulations of cancer cells, only one of
which is sensitive to androgen deprivation, and successfully reproduces the
experimentally observed three phases of \ tumor evolution; however, the issue
of IAS therapy design was not explicitly addressed. \cite{Ideta2008} applied a
hybrid dynamical system approach to model prostate tumor evolution under IAS
and then used it to study the effect of different\ therapy protocols on tumor
growth and time to relapse through numerical and bifurcation analyses. Several
extensions of the works by \cite{Jackson2004} and \cite{Ideta2008} have been
proposed and we briefly review some of them. A nonlinear model was developed
by \cite{Shimada2008} to account for the competition between different cancer
cell subpopulations, while \cite{Tao2010} proposed a model based on switched
ordinary differential equations. The problem of individualized prostate cancer
treatment was formulated as an optimal control problem for which a piecewise
affine system model was developed by \cite{Suzuki2010}. \cite{Hirata2010}
modeled the prostate tumor under IAS as a feedback control system for the
purpose of patient classification, while \cite{Hirata2010B} solved the model
from \cite{Hirata2010} analytically to derive conditions for patient relapse.

Most of the exisiting models provide insights into the dynamics of prostate
cancer evolution under ADT but do not address the issue of therapy design.
Moreover, previous work focusing on classifying patients into groups in order
to infer optimal treatment schemes have been based on more manageable, albeit
less accurate, approaches to nonlinear hybrid dynamical systems. In contrast,
a nonlinear hybrid automaton model was recently proposed by \cite{Liu2015} and
$\delta$-reachability analysis was used to identify patient-specific therapy
schemes. Although this model was shown to be in good agreement with published
clinical data, it did not account for noise and fluctutations that are
inherently associated with cell population dynamics and monitoring of clinical
data. Stochastic effects were incorporated into a hybrid model of tumor growth
under IAS\ therapy by \cite{Tanaka2010}, but the ensuing analysis was
performed considering a pre-determined therapy scheme, i.e., no design of
personalized therapy was carried out. This paper is motivated by the need to
develop optimal personalized IAS\ therapy based on stochastic models of
prostate cancer evolution and represents a first attempt towards this goal
using a Stochastic Hybrid Automaton (SHA) model of cancer progression.

In this paper, we draw upon the deterministic hybrid automaton model from
\cite{Liu2015} to which we incorporate stochastic effects. We propose a cost
metric in terms of the desired outcome of IAS\ therapy that is parameterized
by a controllable parameter vector, and formulate the problem of optimal
personalized therapy design as the search for the parameter values which
minimize our cost metric. We use Infinitesimal Perturbation Analysis (IPA) for
hybrid systems [\cite{Cassandras2010}] to derive gradient estimates of the
cost metric with respect to the controllable vector of interest, which can be
subsequently incorporated into a standard gradient-based optimization
algorithm. Our main focus is on establishing a framework for IPA applications
to biologically-inspired problems which is illustrated here with the case of
prostate cancer therapy design. The advantages of our approach are twofold:
first, we can obtain sensitivity estimates with respect to the various model
parameters from actual data so as to differentiate critical ones from others
that are not. Moreover, IPA efficiently yields sensitivies with respect to
controllable parameters in a therapy (i.e., control policy), which is arguably
the ultimate goal of personalized therapy design.

In Section 2 we formulate the problem of personalized therapy design based on
an SHA model of prostate cancer evolution. Section 3 presents the general
framework of IPA and details the derivation of IPA estimators for therapy
evaluation and optimization. We include final remarks in Section 4.


\section{PROBLEM FORMULATION}

\subsection{SHA Model of Prostate Cancer Progression}

The system we consider comprises a prostate tumor under IAS therapy, which is
modeled as a Stochastic Hybrid Automaton (SHA). We adopt a standard
SHA\ definition [\cite{Cassandras2008}]:%
\[
G_{h}=\left(  Q,X,E,U,f,\phi,Inv,guard,\rho,q_{0},x_{0}\right)
\]
where $Q$ is a set of discrete states; $X$ is a continuous state space; $E$ is
a finite set of events; $U$ is a set of admissible controls; $f$ is a vector
field, $f:Q\times X\times U\rightarrow X$; $\phi$ is a discrete state
transition function, $\phi:Q\times X\times E\rightarrow Q$; $Inv$ is a set
defining an invariant condition (when this condition is violated at some $q\in
Q$, a transition must occur); $guard$ is a set defining a guard condition,
$guard\subseteq Q\times Q\times X$ (when this condition is satisfied at some
$q\in Q$, a transition is allowed to occur); $\rho$ is a reset function,
$\rho:Q\times Q\times X\times E\rightarrow X$; $q_{0}$ is an initial discrete
state; $x_{0}$ is an initial continuous state.

With this framework in place, we define a SHA model of prostate cancer
progression in terms of the following:

1. \emph{Discrete state set }$Q$. Hormone therapy for prostate cancer consists
of administering medical agents that cause androgen suppression in an effort
to decrease the population of prostate cancer cells and hence the size of the
tumor. A common biomarker used to monitor the efficacy of such treatment is
the serum level of Prostate-Specific Antigen (PSA), whose value provides an
estimate of the size of the prostate cancer population.

In IAS therapy, medication is suspended when a sufficient reduction in the
size of the cancer cell populations is achieved. Since population sizes are
not directly observable, this reduction is estimated in terms of the patient's
PSA\ level; hence, the patient goes off therapy once his PSA reaches a lower
threshold value. Similarly, medication is reinstated once the cancer cell
populations have significantly recovered, which corresponds to when the
patient's PSA\ level reaches an upper threshold value. Thus, we define
$Q=\left\{  q^{ON},q^{OFF}\right\}  $, where $q^{ON}$ ($q^{OFF}$,
respectively) is the on-treatment (off-treatment, respectively) operational
mode of the system.

2. \emph{State space }$X$. The continuous state space $X$ is defined in terms
of the biomarkers commonly monitored during IAS therapy, namely the PSA level
and the androgen concentration in the patient's serum. We assume the
coexistence of two subpopulations of cancer cells within the tumor: Hormone
Sensitive Cells (HSCs) and Castration Resistant Cells (CRCs). The
proliferation of the former is negatively affected by hormone therapy, while
the survival rate of the latter decreases in androgen-rich environments.

We thus define a state vector $x(t)=\left[  x_{1}(t),x_{2}(t),x_{3}(t)\right]
$ with $x_{i}(t)\in\mathbb{R}^{+}$, such that $x_{1}(t)$ is the total
population of HSCs, $x_{2}(t)$ is the total population of CRCs, and $x_{3}(t)$
is the concentration of androgen in the serum. Since prostate cancer cells
secrete high levels of PSA, it is frequently assumed that the serum PSA
concentration can be modeled as a linear combination of the cancer cell
subpopulations, i.e., $c_{1}x_{1}(t)+c_{2}x_{2}(t)$. Another common assumption
is that both HSCs and CRCs secrete PSA equivalently, so that $c_{1}=c_{2}=1$
[\cite{Ideta2008}]. In this work, we adopt these assumptions. We also define a
\textquotedblleft clock\textquotedblright\ state variable $z_{i}%
(t)\in\mathbb{R}^{+}$, $i=1,2$, where $z_{1}(t)$ ($z_{2}(t)$, respectively)
measures the time spent by the system in state $q_{ON}$ ($q_{OFF}$,
respectively). The clock is reset to zero once a state transition takes place.
Setting $z(t)=\left[  z_{1}(t),z_{2}(t)\right]  $, the complete state vector
is $\left[  x(t),z(t)\right]  $.

3. \emph{Event set }$E$. We define the SHA event set as $E=\left\{
e_{1},e_{2}\right\}  $, where $e_{1}$ corresponds to the condition $\left[
x_{1}(t)+x_{2}(t)=\theta_{1}\text{ from above}\right]  $ and $e_{2}$
corresponds to $\left[  x_{1}(t)+x_{2}(t)=\theta_{2}\text{ from below}\right]
$.

4. \emph{Admissible control set }$U$. As described earlier, IAS therapy
consists of cycles of androgen deprivation delivered intermittently with
off-treatment periods. The cycles of androgen deprivation are suspended when
the patient's PSA level reaches a lower threshold value, while therapy
recommences once the PSA level reaches an upper threshold value. Hence, an IAS
therapy can be viewed as a controlled process characterized by two parameters:
$\theta=\left[  \theta_{1},\theta_{2}\right]  \in\Theta$, where $\theta_{1}%
\in\left[  \theta_{1}^{\min},\theta_{1}^{\max}\right]  $ is the lower
threshold value of the patient's PSA\ level, and $\theta_{2}\in\left[
\theta_{2}^{\min},\theta_{2}^{\max}\right]  $ is the upper threshold value of
the patient's PSA\ level, with $\theta_{1}^{\max}<\theta_{2}^{\min}$. At any
time $t$, the feasible control set for the IAS therapy controller is
$U=\left\{  0,1\right\}  $ and the control is defined as:%
\begin{equation}
u\left(  x(t),z(t)\right)  \equiv\left\{
\begin{array}
[c]{ll}%
0 & \text{if }x_{1}(t)+x_{2}(t)<\theta_{2}\text{, }q(t)=q^{OFF}\\
1 & \text{if }x_{1}(t)+x_{2}(t)>\theta_{1}\text{, }q(t)=q^{ON}%
\end{array}
\right.  \label{eq: Control law}%
\end{equation}
This is a simple form of hysteresis control to ensure that hormone therapy
will be suspended whenever a patient's PSA\ level drops below a minimum
threshold value, and that therapy will resume whenever a patient's PSA\ level
reaches a maximum threshold value.

5. \emph{System dynamics}. The SHA system dynamics describe the evolution of
continuous state variables over time, as well as the rules for discrete state
transitions. First, the \emph{continuous (time-driven) dynamics} capture the
prostate cancer cell population dynamics, which are defined in terms of their
proliferation, apoptosis, and conversion rates. Existing studies commonly use
Michaelis-Menten-like functions to model the rates of proliferation and
apoptosis [\cite{Ideta2008}, \cite{Jackson2004}, \cite{Jackson2004B}].
Recently \cite{Liu2015} obtained greater consistency between clinical data and
simulated population dynamics by modeling these rates using sigmoid functions.
In what follows, we incorporate stochastic effects into the deterministic
model from \cite{Liu2015} and obtain:%

\begin{equation}%
\begin{array}
[c]{ll}%
\dot{x}_{1}(t) & =\alpha_{1}\left[  1+e^{-\left(  x_{3}(t)-k_{1}\right)
k_{2}}\right]  ^{-1}\cdot x_{1}(t)\\
& -\beta_{1}\left[  1+e^{-\left(  x_{3}(t)-k_{3}\right)  k_{4}}\right]
^{-1}\cdot x_{1}(t)\\
& -\left[  m_{1}\left(  1-\frac{x_{3}(t)}{x_{3,0}}\right)  +\lambda
_{1}\right]  \cdot x_{1}(t)\\
& +\mu_{1}+\zeta_{1}(t)
\end{array}
\label{eq: HSC dynamics}%
\end{equation}

\begin{equation}%
\begin{array}
[c]{ll}%
\dot{x}_{2}(t)= & \left[  \alpha_{2}\left(  1-d\frac{x_{3}(t)}{x_{3,0}%
}\right)  -\beta_{2}\right]  x_{2}(t)\\
& +m_{1}\left(  1-\frac{x_{3}(t)}{x_{3,0}}\right)  x_{1}(t)+\zeta_{2}(t)
\end{array}
\label{eq: CRC dynamics}%
\end{equation}

\begin{equation}
\dot{x}_{3}(t)=\left\{
\begin{array}
[c]{ll}%
-\frac{x_{3}(t)}{\sigma}+\mu_{3}+\zeta_{3}(t) &
\begin{array}
[c]{l}%
\text{if }x_{1}(t)+x_{2}(t)>\theta_{1}\\
\text{and }q(t)=q^{ON}%
\end{array}
\text{ }\\
\frac{x_{3,0}-x_{3}(t)}{\sigma}+\mu_{3}+\zeta_{3}(t) &
\begin{array}
[c]{l}%
\text{if }x_{1}(t)+x_{2}(t)<\theta_{2}\text{ }\\
\text{and }q(t)=q^{OFF}%
\end{array}
\end{array}
\right.  \label{eq: Androgen dynamics}%
\end{equation}%
\begin{align}
\dot{z}_{1}(t)  &  =\left\{
\begin{array}
[c]{ll}%
1 & \text{if }q(t)=q^{ON}\\
0 & \text{otherwise}%
\end{array}
\right. \label{eq: z1 dynamics}\\
z_{1}(t^{+})  &  =%
\begin{array}
[t]{ll}%
0 &
\begin{array}
[c]{l}%
\text{if }x_{1}(t)+x_{2}(t)=\theta_{1}\\
\text{and }q(t)=q^{ON}%
\end{array}
\text{ }%
\end{array}
\nonumber
\end{align}

\begin{align}
\dot{z}_{2}(t)  &  =\left\{
\begin{array}
[c]{ll}%
1 & \text{if }q(t)=q^{OFF}\\
0 & \text{otherwise}%
\end{array}
\right. \label{eq: z2 dynamics}\\
z_{2}(t^{+})  &  =%
\begin{array}
[t]{ll}%
0 &
\begin{array}
[c]{l}%
\text{if }x_{1}(t)+x_{2}(t)=\theta_{2}\\
\text{and }q(t)=q^{OFF}\text{ }%
\end{array}
\end{array}
\nonumber
\end{align}
where $\alpha_{1}$ and $\alpha_{2}$ are the HSC proliferation constant and CRC
proliferation constant, respectively; $\beta_{1}$ and $\beta_{2}$ are the HSC
apoptosis constant and CRC apoptosis constant, respectively; $k_{1}$ through
$k_{4}$ are HSC proliferation and apoptosis exponential constants; $m_{1}$ is
the HSC to CRC conversion constant; $x_{3,0}$ corresponds to the
patient-specific androgen constant; $\sigma$ is the androgen degradation
constant; $\lambda_{1}$ is the HSC basal degradation rate; $\mu_{1}$ and
$\mu_{3}$ are the HSC basal production rate and androgen basal production
rate, respectively. Finally, $\{\zeta_{i}(t)\}$, $i=1,2,3$, are stochastic
processes which we allow to have arbitrary characteristics and only assume
them to be piecewise continuous w.p. 1.

Observe that (\ref{eq: HSC dynamics}) and (\ref{eq: CRC dynamics}) seem to be
independent of the discrete state (mode) $q(t)$; however, their dependence on
$x_{3}(t)$, whose dynamics are affected by mode transitions, implies that
$x_{1}(t)$, $x_{2}(t)$ also change due to such transitions. To make this
behavior explicit, we can solve (\ref{eq: Androgen dynamics}) for $x_{3}(t)$
and substitute this solution into (\ref{eq: HSC dynamics}) and
(\ref{eq: CRC dynamics}), as detailed next.

Consider a sample path of the system over $[0,T]$ and denote the time of
occurrence of the $k$th event (of any type) by $\tau_{k}(\theta)$. Since our
complete system state vector is $\left[  x(t),z(t)\right]  $, we shall denote
the state dynamics over any interevent interval $\left[  \tau_{k}(\theta
),\tau_{k+1}(\theta)\right)  $ as follows:%
\[
\dot{x}_{n}(t)=f_{n,k}^{x}(t)\text{, }\dot{z}_{i}(t)=f_{i,k}^{z}(t)\text{,
}n=1,\ldots,3\text{, }i=1,2
\]
Although we include $\theta$ as an argument in the expressions above to stress
the dependence on the controllable parameter, we will subsequently drop this
for ease of notation as long as no confusion arises.

We start our analysis by assuming $q(t)=q^{ON}$ for $t\in$ $\left[  \tau
_{k},\tau_{k+1}\right)  $. It is clear from (\ref{eq: Androgen dynamics}) that
$\dot{x}_{3}(t)=-\frac{x_{3}(t)}{\sigma}+\mu_{3}+\zeta_{3}(t)$, which implies
that, for $t\in\left[  \tau_{k},\tau_{k+1}\right)  $,%
\[%
\begin{array}
[c]{ll}%
x_{3}(t) & =x_{3}(\tau_{k}^{+})e^{-\left(  t-\tau_{k}\right)  /\sigma}\\
& +e^{-t/\sigma}\cdot\int_{\tau_{k}}^{t}e^{\varepsilon/\sigma}\left[  \mu
_{3}+\zeta_{3}(\varepsilon)\right]  d\varepsilon
\end{array}
\]
For notational simplicity, let%
\begin{equation}
\tilde{\zeta}_{3}(t)=\int_{\tau_{k}}^{t}e^{-\left(  t-\varepsilon\right)
/\sigma}\zeta_{3}(\varepsilon)d\varepsilon\label{eq: zeta3_tilda}%
\end{equation}
and define, for $t\in\left[  \tau_{k},\tau_{k+1}\right)  $,
\begin{equation}%
\begin{array}
[c]{ll}%
h^{ON}\left(  t,\tilde{\zeta}_{3}(t)\right)  & \equiv x_{3}(\tau_{k}%
^{+})e^{-\left(  t-\tau_{k}\right)  /\sigma}\\
& +\mu_{3}\sigma\lbrack1-e^{-\left(  t-\tau_{k}\right)  /\sigma}]+\tilde
{\zeta}_{3}(t)
\end{array}
\label{eq: h_ON}%
\end{equation}
Now let $q(t)=q^{OFF}$ for $t\in$ $\left[  \tau_{k},\tau_{k+1}\right)  $, so
that (\ref{eq: Androgen dynamics}) implies that, for $t\in\left[  \tau
_{k},\tau_{k+1}\right)  $,%
\[%
\begin{array}
[c]{ll}%
x_{3}(t) & =x_{3}(\tau_{k}^{+})e^{-\left(  t-\tau_{k}\right)  /\sigma}\\
& +(\mu_{3}\sigma+x_{3,0})[1-e^{-\left(  t-\tau_{k}\right)  /\sigma}%
]+\tilde{\zeta}_{3}(t)
\end{array}
\]
Similarly as above, we define, for $t\in\left[  \tau_{k},\tau_{k+1}\right)  $,%
\begin{equation}%
\begin{array}
[c]{ll}%
h^{OFF}\left(  t,\tilde{\zeta}_{3}(t)\right)  & \equiv x_{3}(\tau_{k}%
^{+})e^{-\left(  t-\tau_{k}\right)  /\sigma}\\
& +(\mu_{3}\sigma+x_{3,0})[1-e^{-\left(  t-\tau_{k}\right)  /\sigma}%
]+\tilde{\zeta}_{3}(t)
\end{array}
\label{eq: h_OFF}%
\end{equation}
It is thus clear that%
\[
x_{3}(t)=\left\{
\begin{array}
[c]{ll}%
h^{ON}\left(  t,\tilde{\zeta}_{3}(t)\right)  & \text{if }q(t)=q^{ON}\\
h^{OFF}\left(  t,\tilde{\zeta}_{3}(t)\right)  & \text{if }q(t)=q^{OFF}%
\end{array}
\right.
\]
Although we include $\tilde{\zeta}_{3}(t)$ as an argument in (\ref{eq: h_ON}%
)-(\ref{eq: h_OFF}) to stress the dependence on the stochastic process, we
will subsequently drop this for ease of notation as long as no confusion
arises. It is now clear that, by using (\ref{eq: h_ON})-(\ref{eq: h_OFF}) in
(\ref{eq: HSC dynamics})-(\ref{eq: CRC dynamics}), we may rewrite the state
variable dynamics as%
\begin{equation}
\dot{x}_{1}(t)=\left\{
\begin{array}
[c]{l}%
\begin{array}
[c]{l}%
\left\{  \alpha_{1}\left[  1+\phi_{\alpha}^{ON}(t)\right]  ^{-1}-\beta
_{1}\left[  1+\phi_{\beta}^{ON}(t)\right]  ^{-1}\right. \\
\left.  +m_{1}\left(  \frac{h^{ON}\left(  t\right)  }{x_{3,0}}\right)
-(m_{1}+\lambda_{1})\right\}  \cdot x_{1}(t)\\
+\mu_{1}+\zeta_{1}(t)\text{ \ \ \ \ \ \ \ \ \ \ \ \ \ \ \ \ \ \ \ \ \ \ \ \ if
}q(t)=q^{ON}%
\end{array}
\\%
\begin{array}
[c]{l}%
\left\{  \alpha_{1}\left[  1+\phi_{\alpha}^{OFF}(t)\right]  ^{-1}-\beta
_{1}\left[  1+\phi_{\beta}^{OFF}(t)\right]  ^{-1}\right. \\
\left.  +m_{1}\left(  \frac{h^{OFF}\left(  t\right)  }{x_{3,0}}\right)
-(m_{1}+\lambda_{1})\right\}  \cdot x_{1}(t)\\
+\mu_{1}+\zeta_{1}(t)\text{ \ \ \ \ \ \ \ \ \ \ \ \ \ \ \ \ \ \ \ \ \ \ \ \ if
}q(t)=q^{OFF}%
\end{array}
\end{array}
\right.  \label{eq: x1 dynamics}%
\end{equation}
and%
\begin{equation}
\dot{x}_{2}(t)=\left\{
\begin{array}
[c]{l}%
\begin{array}
[c]{l}%
\left[  \alpha_{2}\left(  1-d\frac{h^{ON}\left(  t\right)  }{x_{3,0}}\right)
-\beta_{2}\right]  x_{2}(t)\\
+m_{1}\left(  1-\frac{h^{ON}\left(  t\right)  }{x_{3,0}}\right)
x_{1}(t)+\zeta_{2}(t)\\
\text{ \ \ \ \ \ \ \ \ \ \ \ \ \ \ \ \ \ \ \ \ \ \ \ \ \ \ \ \ if }q(t)=q^{ON}%
\end{array}
\\%
\begin{array}
[c]{l}%
\left[  \alpha_{2}\left(  1-d\frac{h^{OFF}\left(  t\right)  }{x_{3,0}}\right)
-\beta_{2}\right]  x_{2}(t)\\
+m_{1}\left(  1-\frac{h^{OFF}\left(  t\right)  }{x_{3,0}}\right)
x_{1}(t)+\zeta_{2}(t)\\
\text{ \ \ \ \ \ \ \ \ \ \ \ \ \ \ \ \ \ \ \ \ \ \ \ \ \ \ \ if }q(t)=q^{OFF}%
\end{array}
\end{array}
\right.  \label{eq: x2 dynamics}%
\end{equation}
with%
\begin{align*}
\phi_{\alpha}^{ON}(t)  &  =e^{-\left(  h^{ON}\left(  t\right)  -k_{1}\right)
k_{2}}\\
\phi_{\beta}^{ON}(t)  &  =e^{-\left(  h^{ON}\left(  t\right)  -k_{3}\right)
k_{4}}\\
\phi_{\alpha}^{OFF}(t)  &  =e^{-\left(  h^{OFF}\left(  t\right)
-k_{1}\right)  k_{2}}\\
\phi_{\beta}^{OFF}(t)  &  =e^{-\left(  h^{OFF}\left(  t\right)  -k_{3}\right)
k_{4}}%
\end{align*}

The \emph{discrete (event-driven) dynamics} are dictated by the occurrence of
events that cause state transitions. Based on the event set $E=\left\{
e_{1},e_{2}\right\}  $ we have defined, the occurrence of $e_{1}$ results in a
transition from $q^{ON}$ to $q^{OFF}$ and the occurrence of $e_{2}$ results in
a transition from $q^{OFF}$ to $q^{ON}$.

\subsection{IAS\ Therapy Evaluation and Optimization}

The effect of an IAS\ therapy $u\left(  \theta,t\right)  $ depends on the
controllable parameter vector $\theta$, as in (\ref{eq: Control law}), and can
be quantified in terms of performance metrics of the form $J\left[
\mathbf{u}(\theta,t)\right]  $. Evaluating $J\left[  \mathbf{u}(\theta
,t)\right]  $ over all possible values of $\theta$ is clearly infeasible.
However, there exist very efficient ways to accomplish this goal for
stochastic hybrid systems. In particular, Perturbation Analysis (PA) is a
methodology to efficiently estimate the sensitivity of the performance with
respect to $\theta$, i.e., when $\theta$ is a real-valued scalar, the
derivative $dJ/d\theta$. This is accomplished by extracting data from a sample
path (simulated or actual) of the observed system based on which an unbiased
estimate of $dJ/d\theta$ can indeed be obtained. The attractive feature of PA
is that the resulting estimates are extracted from a single sample path in a
non-intrusive manner and the computational cost of doing so is, in most cases
of interest, minimal [\cite{Cassandras2008}]. This is in contrast to the the
conventional finite difference estimate of $dJ/d\theta$ obtained through
$[J(\theta+\Delta)-J(\theta)]/\Delta$. Thus, for a vector $\theta$ of
dimension $N$, estimating the gradient $\nabla J\left(  \theta\right)  $
requires a single sample path (with some overhead) instead of $N+1$ sample
paths. The simplest family of PA estimators is Infinitesimal Perturbation
Analysis (IPA). For the SHA model we consider here, IPA has been recently
shown to provide unbiased gradient estimates [\cite{Cassandras2010}] for
virtually arbitrary stochastic hybrid systems. Our goal, therefore, is to
adapt IPA estimators of the form $dJ[\mathbf{u}(\theta,t)]/d\theta$ for
different therapies $\mathbf{u}(\theta,t)$, thus estimating their effects, and
to ultimately show that optimal therapy schemes can be designed by solving
problems of the form $\min_{\theta\in\Theta}$ $J[\mathbf{u}(\theta,t)]$.

We define a sample function in terms of complementary measures of therapy
success. In particular, we take the most adequate treatment schemes to be
those that \textit{(i)} ensure PSA levels are kept as low as possible;
\textit{(ii)} reduce the frequency of on and off-treatment cycles. There is an
obvious trade-off between \textit{(i)} and the cost associated with
\textit{(ii)}, which is related to the duration of the therapy and could
potentially include fixed set up costs incurred when therapy is reinstated.
For simplicity, we disconsider fixed set up costs and take \textit{(ii)} to be
linearly proportional to the length of the on-treatment cycles. We thus
associate weights $W_{i}$, $i=1,2$, with \textit{(i)} and \textit{(ii)},
respectively, and define our sample function as the sum of the average PSA
level and the average duration of an on-treatment cycle over a fixed time
interval $\left[  0,T\right]  $. We also normalize our sample function to
ensure that the trade-off between \textit{(i)} and \textit{(ii)} is captured
appropriately: we divide \textit{(i)} by the value of the patient's PSA level
at the start of the first on-treatment cycle ($PSA_{init}$), and note that
\textit{(ii)} is naturally normalized by $T$. Recall that the total population
size of prostate cancer cells is assumed to reflect the serum PSA
concentration, and that we have defined clock variables which measure the time
elapsed in each of the treatment modes, so that our sample function can be
written as%
\begin{equation}%
\begin{array}
[c]{ll}%
L\left(  \theta,x(0),z(0),T\right)  & =\frac{1}{T}\left[  W_{1}\int_{0}%
^{T}\left[  \frac{x_{1}(\theta,t)+x_{2}(\theta,t)}{PSA_{init}}\right]
dt\right. \\
& +\left.  W_{2}\int_{0}^{T}z_{1}(t)dt\right]
\end{array}
\label{eq: Sample fn}%
\end{equation}
where $x(0)$ and $z(0)$ are given initial conditions. We can then define the
overall performance metric as%
\begin{equation}
J\left(  \theta,x(0),z(0),T\right)  =E\left[  L\left(  \theta
,x(0),z(0),T\right)  \right]  \label{eq: Performance fn}%
\end{equation}

Hence, the problem of determining the optimal IAS therapy can be formulated as%
\begin{equation}
\min_{\theta\in\Theta}E\left[  L\left(  \theta,x(0),z(0),T\right)  \right]
\label{eq: Opt problem}%
\end{equation}

\section{INFINITESIMAL PERTURBATION ANALYSIS}

Consider a sample path generated by the SHA\ over $[0,T]$ and recall that we
have defined $\tau_{k}(\theta)$ to be the time of occurrence of the $k$th
event (of any type), where $\theta$ is a controllable parameter vector of
interest. We denote the state and event time derivatives with respect to
parameter $\theta$ as $x^{\prime}(\theta,t)\equiv\frac{dx(\theta,t)}{d\theta}$
and $\tau_{k}^{\prime}(\theta)\equiv\frac{d\tau_{k}(\theta)}{d\theta}$,
respectively, for $k=1,\ldots,N$. As mentioned earlier, the dynamics of
$x(\theta,t)$ are fixed over any interevent interval $\left[  \tau_{k}%
(\theta),\tau_{k+1}(\theta)\right]  $ and we write $\dot{x}(t)=f_{k}\left(
\theta,x,t\right)  $ to represent the state dynamics over this interval.
Although we include $\theta$ as an argument in the expressions above to stress
the dependence on the controllable parameter, we will subsequently drop this
for ease of notation as long as no confusion arises. It is shown in
\cite{Cassandras2010} that the state derivative satisfies%
\begin{equation}
\frac{d}{dt}x^{\prime}(t)=\frac{df_{k}(t)}{dx}x^{\prime}(t)+\frac{df_{k}%
(t)}{d\theta} \label{eq: State deriv}%
\end{equation}
with the following boundary condition:%
\begin{equation}
x^{\prime}(\tau_{k}^{+})=x^{\prime}(\tau_{k}^{-})+\left[  f_{k-1}(\tau_{k}%
^{-})-f_{k}(\tau_{k}^{+})\right]  \cdot\tau_{k}^{\prime}
\label{eq: Boundary cond}%
\end{equation}
We note that (\ref{eq: Boundary cond}) is valid when $x(\theta,t)$ is
continuous in $t$ at $t=\tau_{k}$. If this is not the case and the value of
$x(\tau_{k}^{+})$ is determined by the reset function $\rho\left(
q,q^{\prime},x,e\right)  $, then%
\begin{equation}
x^{\prime}(\tau_{k}^{+})=\frac{d\rho\left(  q,q^{\prime},x,e\right)  }%
{d\theta} \label{eq: Reset fn.}%
\end{equation}
Knowledge of $\tau_{k}^{\prime}$ is, therefore, needed in order to evaluate
(\ref{eq: Boundary cond}). Following the framework in \cite{Cassandras2010},
there are three types of events for a general stochastic hybrid system:
\emph{(i) Exogenous Events}. These events cause a discrete state transition
independent of $\theta$ and satisfy $\tau_{k}^{\prime}=0$. \emph{(ii)
Endogenous Events}. Such an event occurs at time $\tau_{k}$ if there exists a
continuously differentiable function $g_{k}:\mathbb{R}^{n}\times
\Theta\rightarrow\mathbb{R}$ such that $\tau_{k}=\min\left\{  t>\tau
_{k-1}:g_{k}\left(  x(\theta,t),\theta\right)  =0\right\}  $, where the
function $g_{k}$ usually corresponds to a guard condition in a hybrid
automaton. Taking derivatives with respect to $\theta$, it is straighforward
to obtain%
\begin{equation}
\tau_{k}^{\prime}=-\left[  \frac{dg_{k}}{dx}\cdot f_{k-1}(\tau_{k}%
^{-})\right]  ^{-1}\cdot\left(  \frac{dg_{k}}{d\theta}+\frac{dg_{k}}{dx}\cdot
x^{\prime}(\tau_{k}^{-})\right)  \label{eq: Event time deriv}%
\end{equation}
as long as $\frac{dg_{k}}{dx}\cdot f_{k-1}(\tau_{k}^{-})\neq0$. \emph{(iii)
Induced Events}. Such an event occurs at time $\tau_{k}$ if it is triggered by
the occurrence of another event at time $\tau_{m}\leq\tau_{k}$ (details can be
found in \cite{Cassandras2010}).

Returning to our problem of personalized prostate cancer therapy design, we
define the derivatives of the states $x_{n}(\theta,t)$ and $z_{j}(\theta,t)$
and event times $\tau_{k}(\theta)$ with respect to $\theta_{i}$, $i,j=1,2$,
$n=1,\ldots,3$, as follows:%
\begin{equation}
x_{n,i}^{\prime}(t)\equiv\frac{\partial x_{n}(\theta,t)}{\partial\theta_{i}%
}\text{, \ }z_{j,i}^{\prime}(t)\equiv\frac{\partial z_{j}(\theta,t)}%
{\partial\theta_{i}},\text{ }\tau_{k,i}^{\prime}\equiv\frac{\partial\tau
_{k}(\theta)}{\partial\theta_{i}} \label{eq: Deriv def}%
\end{equation}
Our goal is to obtain an estimate of $\nabla J(\theta)$ by evaluating the
sample gradient $\nabla L(\theta)$, which is computed using data extracted
from a sample path of the system (e.g., by simulating a sample path of our
SHA\ model using clinical data). We will assume that the derivatives
$dL/d\theta_{i}$ exist w.p. 1 for all $\theta_{i}\in\mathbb{R}^{+}$. It is
also easy to verify that $L\left(  \theta\right)  $ is Lipschitz continuous
for $\theta_{i}\in\mathbb{R}^{+}$. We will further assume that no two events
can occur at the same time w.p.1. Under these conditions, it has been shown in
\cite{Cassandras2010} that $dL/d\theta_{i}$ is an \textit{unbiased} estimator
of $dJ/d\theta_{i}$, $i=1,2$.

In what follows, we derive the IPA state and event time derivatives for the
events indentified in our SHA model of prostate cancer progression.

\subsection{State and Event Time Derivatives}

We begin by analyzing the state evolution considering each of the states
($q^{ON}$ and $q^{OFF}$) and events ($e_{1}$ and $e_{2}$) defined for our SHA\ model.

1. $q(t)=q^{ON}$ \emph{for }$t\in$\emph{ }$\left[  \tau_{k},\tau_{k+1}\right)
$. Using (\ref{eq: State deriv}) for $x_{1}(t)$, we have, for $i=1,2$,%
\[%
\begin{array}
[c]{ll}%
\frac{d}{dt}x_{1,i}^{\prime}(t) & =\frac{\partial f_{k}^{x_{1}}(t)}{\partial
x_{1}}x_{1}^{\prime}(t)+\frac{\partial f_{k}^{x_{1}}(t)}{\partial x_{2}}%
x_{2}^{\prime}(t)+\frac{\partial f_{k}^{x_{1}}(t)}{\partial x_{3}}%
x_{3}^{\prime}(t)\\
& +\frac{\partial f_{k}^{x_{1}}(t)}{\partial z_{1}}z_{1}^{\prime}%
(t)+\frac{\partial f_{k}^{x_{1}}(t)}{\partial z_{2}}z_{2}^{\prime}%
(t)+\frac{\partial f_{k}^{x_{1}}(t)}{\partial\theta_{i}}%
\end{array}
\]
From (\ref{eq: x1 dynamics}), we have $\frac{\partial f_{k}^{x_{1}}%
(t)}{\partial x_{n}}=\frac{\partial f_{k}^{x_{1}}(t)}{\partial z_{i}}%
=\frac{\partial f_{k}^{x_{1}}(t)}{\partial\theta_{i}}=0$, $n=2,3$, $i=1,2$,
and%
\begin{equation}%
\begin{array}
[c]{ll}%
\frac{\partial f_{k}^{x_{1}}(t)}{\partial x_{1}} & =\alpha_{1}\left[
1+\phi_{\alpha}^{ON}(t)\right]  ^{-1}-\beta_{1}\left[  1+\phi_{\beta}%
^{ON}(t)\right]  ^{-1}\\
& -m_{1}\left(  1-\frac{h^{ON}\left(  t\right)  }{x_{3,0}}\right)
-\lambda_{1}%
\end{array}
\label{eq: Partial x1}%
\end{equation}
Combining the last two equations and solving for $x_{1,i}^{\prime}(t)$, we
obtain%
\begin{equation}
x_{1,i}^{\prime}(t)=x_{1,i}^{\prime}(\tau_{k}^{+})e^{A\left(  t\right)
}\text{, \ \ }t\in\left[  \tau_{k},\tau_{k+1}\right)  \label{eq: x1_prime ON}%
\end{equation}
and, in particular,%
\begin{equation}
x_{1,i}^{\prime}(\tau_{k+1}^{-})=x_{1,i}^{\prime}(\tau_{k}^{+})e^{A\left(
\tau_{k}\right)  }\label{eq: Deriv x1 t1}%
\end{equation}
with%
\[%
\begin{array}
[c]{l}%
A\left(  \tau_{k}\right)  \equiv\int_{\tau_{k}}^{\tau_{k+1}}\left[
\frac{\alpha_{1}}{1+\phi_{\alpha}^{ON}(t)}-\frac{\beta_{1}}{1+\phi_{\beta
}^{ON}(t)}\right]  dt\\
\text{ \ }-\int_{\tau_{k}}^{\tau_{k+1}}\frac{m_{1}}{x_{3,0}}h^{ON}\left(
t\right)  dt-\left(  m_{1}+\lambda_{1}\right)  \left(  \tau_{k+1}-\tau
_{k}\right)
\end{array}
\]
Similarly for $x_{2}(t)$, we have from (\ref{eq: x2 dynamics}) that
$\frac{\partial f_{k}^{x_{2}}(t)}{\partial x_{3}}=\frac{\partial f_{k}^{x_{2}%
}(t)}{\partial z_{i}}=\frac{\partial f_{k}^{x_{2}}(t)}{\partial\theta_{i}}=0$,
$i=1,2$, and%
\begin{equation}%
\begin{array}
[c]{l}%
\frac{\partial f_{k}^{x_{2}}(t)}{\partial x_{1}}=m_{1}\left(  1-\frac
{h^{ON}\left(  t\right)  }{x_{3,0}}\right)  \\
\frac{\partial f_{k}^{x_{2}}(t)}{\partial x_{2}}=\alpha_{2}\left(
1-d\frac{h^{ON}\left(  t\right)  }{x_{3,0}}\right)  -\beta_{2}%
\end{array}
\label{eq: Partial x2}%
\end{equation}
Combining the last two equations and solving for $x_{2,i}^{\prime}(t)$ yields,
for $t\in\left[  \tau_{k},\tau_{k+1}\right)  $,%
\begin{equation}
x_{2,i}^{\prime}(t)=x_{2,i}^{\prime}(\tau_{k}^{+})e^{B_{1}(t)}+B_{2}\left(
t,x_{1,i}^{\prime}(\tau_{k}^{+}),A\left(  t\right)  \right)
\label{eq: x2_prime ON}%
\end{equation}
and, in particular,%
\begin{equation}
x_{2,i}^{\prime}(\tau_{k+1}^{-})=x_{2,i}^{\prime}(\tau_{k}^{+})e^{B_{1}%
(\tau_{k})}+B_{2}\left(  \tau_{k},x_{1,i}^{\prime}(\tau_{k}^{+}),A\left(
\tau_{k}\right)  \right)  \label{eq: Deriv x2 t1}%
\end{equation}
with%
\begin{align*}
&  B_{1}\left(  \tau_{k}\right)  \equiv\int_{\tau_{k}}^{\tau_{k+1}}\left[
\alpha_{2}\left(  1-d\frac{h^{ON}\left(  t\right)  }{x_{3,0}}\right)
-\beta_{2}\right]  dt\\
&  B_{2}\left(  \cdot\right)  \equiv e^{B_{1}(\tau_{k})}\int_{\tau_{k}}%
^{\tau_{k+1}}G_{1}\left(  t,\tau_{k}\right)  e^{-B_{1}(\tau_{k})}dt
\end{align*}
where $G_{1}\left(  t,\tau_{k}\right)  =m_{1}\left(  1-\frac{h^{ON}\left(
t\right)  }{x_{3,0}}\right)  x_{1,i}^{\prime}(\tau_{k}^{+})e^{A\left(
t\right)  }$, $t\in\left[  \tau_{k},\tau_{k+1}\right)  $.

In the case of $x_{3}(t)$, it is clear from (\ref{eq: Androgen dynamics}) that
$\frac{\partial f_{k}^{x_{3}}(t)}{\partial x_{i}}=\frac{\partial f_{k}^{x_{3}%
}(t)}{\partial z_{i}}=\frac{\partial f_{k}^{x_{3}}(t)}{\partial\theta_{i}}=0$,
$i=1,2$, and $\frac{\partial f_{k}^{x_{3}}(t)}{\partial x_{3}}=-\frac
{1}{\sigma}$. Substituting in (\ref{eq: State deriv}) and solving for
$x_{3,i}^{\prime}(t)$, for $i=1,2$, yields $x_{3,i}^{\prime}(t)=x_{3,i}%
^{\prime}(\tau_{k}^{+})e^{-\left(  t-\tau_{k}\right)  /\sigma}$,
$t\in\emph{\ }\left[  \tau_{k},\tau_{k+1}\right)  $, and, in particular,%
\begin{equation}
x_{3,i}^{\prime}(\tau_{k+1}^{-})=x_{3,i}^{\prime}(\tau_{k}^{+})e^{-\left(
t-\tau_{k}\right)  /\sigma} \label{eq: Deriv x3 t1}%
\end{equation}
Finally, in the case of the "clock" state variable $z_{i}(\theta,t)$, $i=1,2$,
based on (\ref{eq: z1 dynamics}) and (\ref{eq: z2 dynamics}), we have
$\frac{\partial f_{k}^{z_{i}}(t)}{\partial x_{n}}=\frac{\partial f_{k}^{z_{i}%
}(t)}{\partial z_{i}}=\frac{\partial f_{k}^{z_{i}}(t)}{\partial\theta_{i}}=0$,
$n=1,\ldots,3$, $i=1,2$, so that $\frac{d}{dt}z^{\prime}(t)=0$ for
$t\in\left[  \tau_{k},\tau_{k+1}\right)  $. This means that the value of the
state derivative of the "clock" variable remains unaltered while $q(t)=q^{ON}%
$, i.e., $z_{i}^{\prime}(t)=z_{i}^{\prime}(\tau_{k}^{+})$, $t\in\left[
\tau_{k},\tau_{k+1}\right)  $.

2. $q(t)=q^{OFF}$ \emph{for }$t\in$\emph{ }$\left[  \tau_{k},\tau
_{k+1}\right)  $. Starting with $x_{1}(t)$, a similar reasoning as above
applies, but now we have%
\[%
\begin{array}
[c]{l}%
\frac{\partial f_{k}^{x_{1}}(t)}{\partial x_{1}}=\alpha_{1}\left[
1+\phi_{\alpha}^{OFF}(t)\right]  ^{-1}-\beta_{1}\left[  1+\phi_{\beta}%
^{OFF}(t)\right]  ^{-1}\\
\text{ \ \ \ \ \ \ \ \ }-m_{1}\left(  1-\frac{h^{OFF}\left(  t\right)
}{x_{3,0}}\right)  -\lambda_{1}%
\end{array}
\]
As a result, (\ref{eq: State deriv}) implies that, for $i=1,2$,%
\begin{equation}
x_{1,i}^{\prime}(t)=x_{1,i}^{\prime}(\tau_{k}^{+})e^{C\left(  t\right)
}\text{, \ }t\in\left[  \tau_{k},\tau_{k+1}\right)  \label{eq: x1_prime OFF}%
\end{equation}
and, in particular,%
\begin{equation}
x_{1,i}^{\prime}(\tau_{k+1}^{-})=x_{1,i}^{\prime}(\tau_{k}^{+})e^{C\left(
\tau_{k}\right)  } \label{eq: Deriv x1 t3minus}%
\end{equation}
with%

\[%
\begin{array}
[c]{l}%
C\left(  \tau_{k}\right)  \equiv\int_{\tau_{k}}^{\tau_{k+1}}\left[
\frac{\alpha_{1}}{1+\phi_{\alpha}^{OFF}(t)}-\frac{\beta_{1}}{1+\phi_{\beta
}^{OFF}(t)}\right]  dt\\
\text{ }-\int_{\tau_{k}}^{\tau_{k+1}}\frac{m_{1}}{x_{3,0}}h^{OFF}\left(
t\right)  dt-\left(  m_{1}+\lambda_{1}\right)  \left(  \tau_{k+1}-\tau
_{k}\right)
\end{array}
\]
Similarly for $x_{2}(t)$, we have%
\[%
\begin{array}
[c]{l}%
\frac{\partial f_{k}^{x_{2}}(t)}{\partial x_{1}}=m_{1}\left(  1-\frac
{h^{OFF}\left(  t\right)  }{x_{3,0}}\right)  \\
\frac{\partial f_{k}^{x_{2}}(t)}{\partial x_{2}}=\alpha_{2}\left(
1-d\frac{h^{OFF}\left(  t\right)  }{x_{3,0}}\right)  -\beta_{2}%
\end{array}
\]
It is thus straightforward to verify that (\ref{eq: State deriv}) yields, for
$t\in\left[  \tau_{k},\tau_{k+1}\right)  $,%
\begin{equation}
x_{2,i}^{\prime}(t)=x_{2,i}^{\prime}(\tau_{k}^{+})e^{D_{1}(t)}+D_{2}\left(
t,x_{1,i}^{\prime}(\tau_{k}^{+}),C\left(  t\right)  \right)
\label{eq: x2_prime OFF}%
\end{equation}
and, in particular,%
\begin{equation}
x_{2,i}^{\prime}(\tau_{k+1}^{-})=x_{2,i}^{\prime}(\tau_{k}^{+})e^{D_{1}%
(\tau_{k})}+D_{2}\left(  \tau_{k},x_{1,i}^{\prime}(\tau_{k}^{+}),C\left(
\tau_{k}\right)  \right)  \label{eq: Deriv x2 tkminus}%
\end{equation}
with%
\begin{align*}
&  D_{1}\left(  \tau_{k}\right)  \equiv\int_{\tau_{k}}^{\tau_{k+1}}\left[
\alpha_{2}\left(  1-d\frac{h^{OFF}\left(  t\right)  }{x_{3,0}}\right)
-\beta_{2}\right]  dt\\
&  D_{2}\left(  \cdot\right)  \equiv e^{D_{1}(\tau_{k})}\int_{\tau_{k}}%
^{\tau_{k+1}}G_{2}\left(  t,\tau_{k}\right)  e^{-D_{1}(\tau_{k})}dt
\end{align*}
where $G_{2}\left(  t,\tau_{k}\right)  =m_{1}\left(  1-\frac{h^{OFF}\left(
t\right)  }{x_{3,0}}\right)  x_{1,i}^{\prime}(\tau_{k}^{+})e^{C\left(
t\right)  }$, $t\in\left[  \tau_{k},\tau_{k+1}\right)  $.

In the case of $x_{3}(t)$, we will have $\frac{\partial f_{k}^{x_{3}}%
(t)}{\partial x_{i}}=\frac{\partial f_{k}^{x_{3}}(t)}{\partial z_{i}}%
=\frac{\partial f_{k}^{x_{3}}(t)}{\partial\theta_{i}}=0$, $i=1,2$, and
$\frac{\partial f_{k}^{x_{3}}(t)}{\partial x_{3}}=-\frac{1}{\sigma}$, so that
(\ref{eq: State deriv}) implies $x_{3,i}^{\prime}(t)=x_{3,i}^{\prime}(\tau
_{k}^{+})e^{-\left(  t-\tau_{k}\right)  /\sigma}$, $t\in\emph{\ }\left[
\tau_{k},\tau_{k+1}\right)  $, and, in particular,
\begin{equation}
x_{3,i}^{\prime}(\tau_{k+1}^{-})=x_{3,i}^{\prime}(\tau_{k}^{+})e^{-\left(
t-\tau_{k}\right)  /\sigma} \label{eq: Deriv x3 t3minus}%
\end{equation}
Finally, in the case of the "clock" state variable $z_{i}(\theta,t)$, $i=1,2$,
based on (\ref{eq: z1 dynamics}) and (\ref{eq: z2 dynamics}), we have
$\frac{\partial f_{k}^{z_{i}}(t)}{\partial x_{n}}=\frac{\partial f_{k}^{z_{i}%
}(t)}{\partial z_{i}}=\frac{\partial f_{k}^{z_{i}}(t)}{\partial\theta_{i}}=0$,
$n=1,\ldots,3$, $i=1,2$, so that $\frac{d}{dt}z^{\prime}(t)=0$ for
$t\in\left[  \tau_{k},\tau_{k+1}\right)  $. This means that the value of the
state derivative of the "clock" variable remains unaltered while
$q(t)=q^{OFF}$, i.e., $z_{i}^{\prime}(t)=z_{i}^{\prime}(\tau_{k}^{+})$,
$t\in\left[  \tau_{k},\tau_{k+1}\right)  $.

3. \emph{A state transition from }$q^{ON}$ \emph{to }$q^{OFF}$ \emph{takes
place at time }$\tau_{k}$. This necessarily implies that event $e_{1}$
occurred at time $\tau_{k}$. From (\ref{eq: Boundary cond}) we have, for
$i=1,2$,%
\begin{equation}
x_{1,i}^{\prime}(\tau_{k}^{+})=x_{1,i}^{\prime}(\tau_{k}^{-})+\left[
f_{k}^{x_{1}}(\tau_{k}^{-})-f_{k+1}^{x_{1}}(\tau_{k}^{+})\right]  \cdot
\tau_{k,i}^{\prime} \label{eq: x1_prime ON/OFF}%
\end{equation}

Observe that, from (\ref{eq: x1 dynamics}), $f_{k}^{x_{1}}(\tau_{k}^{-})$ and
$f_{k+1}^{x_{1}}(\tau_{k}^{+})$ ultimately depend on $h^{ON}\left(  \tau
_{k}^{-}\right)  $ and $h^{OFF}\left(  \tau_{k}^{+}\right)  $, respectively.
Also, a transition from $q^{ON}$ to $q^{OFF}$ at time $\tau_{k}$ implies that
$q(t)=q^{ON}$, $t\in\left[  \tau_{k-1},\tau_{k}\right)  $ and $q(t)=q^{OFF}$,
$t\in\left[  \tau_{k},\tau_{k+1}\right)  $. Hence, evaluating $h^{ON}\left(
\tau_{k}^{-}\right)  $ from (\ref{eq: h_ON}) over the appropriate time
interval results in%
\[%
\begin{array}
[c]{ll}%
h^{ON}\left(  \tau_{k}^{-}\right)  & =x_{3}(\tau_{k-1}^{+})e^{-\left(
\tau_{k}-\tau_{k-1}\right)  /\sigma}\\
& +\mu_{3}\sigma\lbrack1-e^{-\left(  \tau_{k}-\tau_{k-1}\right)  /\sigma
}]+\tilde{\zeta}_{3}(\tau_{k})
\end{array}
\]
and it follows directly from (\ref{eq: h_OFF}) that $h^{OFF}\left(  \tau
_{k}^{+}\right)  =x_{3}(\tau_{k}^{+})$.

Furthermore, by continuity of $x_{n}(t)$ (due to conservation of mass),
$x_{n}(\tau_{k}^{+})=x_{n}(\tau_{k}^{-})$, $n=1,2,3$. Also, since we have
assumed that $\left\{  \zeta_{i}(t)\right\}  $, $i=1,2,3$, is piecewise
continuous w.p.1 and that no two events can occur at the same time w.p.1,
$\zeta_{i}(\tau_{k}^{-})=$ $\zeta_{i}(\tau_{k}^{+})$, $i=1,2,3$. Hence, by
evaluating $\Delta_{f}^{1}\left(  \tau_{k}\right)  \equiv f_{k}^{x_{1}}%
(\tau_{k}^{-})-f_{k+1}^{x_{1}}(\tau_{k}^{+})$ we obtain%
\begin{equation}%
\begin{array}
[c]{l}%
\Delta_{f}^{1}\left(  \tau_{k},\zeta_{3}\left(  \tau_{k}\right)  \right)
=\left\{  \alpha_{1}\left[  1+\phi_{\alpha}^{ON}(\tau_{k}^{-})\right]
^{-1}\right. \\
\text{ }-\alpha_{1}\left[  1+\phi_{\alpha}^{OFF}(\tau_{k}^{+})\right]
^{-1}-\beta_{1}\left[  1+\phi_{\beta}^{ON}(\tau_{k}^{-})\right]  ^{-1}\\
\text{ }+\beta_{1}\left[  1+\phi_{\beta}^{OFF}(\tau_{k}^{+})\right]  ^{-1}\\
\text{ }\left.  +\frac{m_{1}}{x_{3,0}}\left[  h^{ON}\left(  \tau_{k}%
^{-}\right)  -x_{3}(\tau_{k})\right]  \right\}  \cdot x_{1}(\tau_{k})
\end{array}
\label{eq: deltaF1 ON/OFF}%
\end{equation}

Finally, the term $\tau_{k,i}^{\prime}$, which corresponds to the event time
derivative with respect to $\theta_{i}$ at event time $\tau_{k}$, is
determined using (\ref{eq: Event time deriv}), as will be detailed in Lemma 1 later.

A similar analysis applies to $x_{2}(t)$, so that, for $i=1,2$,%
\begin{equation}
x_{2,i}^{\prime}(\tau_{k}^{+})=x_{2,i}^{\prime}(\tau_{k}^{-})+\left[
f_{k}^{x_{2}}(\tau_{k}^{-})-f_{k+1}^{x_{2}}(\tau_{k}^{+})\right]  \cdot
\tau_{k,i}^{\prime} \label{eq: x2_prime ON/OFF}%
\end{equation}
where $\tau_{k,i}^{\prime}$ will be derived in Lemma 1, and $f_{k}^{x_{2}%
}(\tau_{k}^{-})$ and $f_{k+1}^{x_{2}}(\tau_{k}^{+})$ ultimately depend on
$h^{ON}\left(  \tau_{k}^{-}\right)  $ and $h^{OFF}\left(  \tau_{k}^{+}\right)
$, respectively. Hence, evaluating $\Delta_{f}^{2}\left(  \tau_{k}\right)
\equiv f_{k}^{x_{2}}(\tau_{k}^{-})-f_{k+1}^{x_{2}}(\tau_{k}^{+})$ from
(\ref{eq: x2 dynamics}) yields%
\begin{equation}%
\begin{array}
[c]{ll}%
\Delta_{f}^{2}\left(  \tau_{k},\zeta_{3}\left(  \tau_{k}\right)  \right)  &
=\frac{\alpha_{2}d}{x_{3,0}}\left[  x_{3}(\tau_{k})-h^{ON}\left(  \tau_{k}%
^{-}\right)  \right]  \cdot x_{2}(\tau_{k})\\
& -\frac{m_{1}}{x_{3,0}}\left[  h^{ON}\left(  \tau_{k}^{-}\right)  -x_{3}%
(\tau_{k})\right]  \cdot x_{1}(\tau_{k})
\end{array}
\label{eq: deltaF2 ON/OFF}%
\end{equation}

Finally, for $x_{3}(t)$, (\ref{eq: Boundary cond}) can be easily seen to
yield, for $i=1,2$,%

\[
x_{3,i}^{\prime}(\tau_{k}^{+})=x_{3,i}^{\prime}(\tau_{k}^{-})-\frac{x_{3,0}%
}{\sigma}\cdot\tau_{k,i}^{\prime}%
\]
In the case of the "clock" state variable, $z_{1}(t)$ is discontinuous in $t$
at $t=\tau_{k}$, while $z_{2}(t)$ is continuous. Applying (\ref{eq: Reset fn.}%
) to the reset function defined in (\ref{eq: z1 dynamics}) yields
$z_{1,i}^{\prime}(\tau_{k}^{+})=0$, $i=1,2$, i.e., the value of $z_{1,i}%
^{\prime}(t)$ is reset to zero whenever event $e_{1}$ takes place. Based on
(\ref{eq: Boundary cond}) and (\ref{eq: z2 dynamics}), it is simple to verify
that, for $i=1,2$,%
\[
z_{2,i}^{\prime}(\tau_{k}^{+})=z_{2,i}^{\prime}(\tau_{k}^{-})-\tau
_{k,i}^{\prime}%
\]

4. \emph{A state transition from }$q^{OFF}$ \emph{to }$q^{ON}$ \emph{takes
place at time }$\tau_{k}$. This necessarily implies that event $e_{2}$
occurred at time $\tau_{k}$. The same reasoning as above applies and from
(\ref{eq: Boundary cond}) we have, for $i=1,2$,%
\begin{equation}
x_{1,i}^{\prime}(\tau_{k}^{+})=x_{1,i}^{\prime}(\tau_{k}^{-})+\left[
f_{k}^{1}(\tau_{k}^{-})-f_{k+1}^{1}(\tau_{k}^{+})\right]  \cdot\tau
_{k,i}^{\prime} \label{eq: x1_prime OFF/ON}%
\end{equation}
where $f_{k}^{1}(\tau_{k}^{-})-f_{k+1}^{1}(\tau_{k}^{+})$ can be evaluated
from (\ref{eq: x1 dynamics}) and ultimately depends on $h^{OFF}\left(
\tau_{k}^{-}\right)  $ and $h^{ON}\left(  \tau_{k}^{+}\right)  $. Recall that
a transition from $q^{OFF}$ to $q^{ON}$ at time $\tau_{k}$ implies that
$q(t)=q^{OFF}$, $t\in\left[  \tau_{k-1},\tau_{k}\right)  $ and $q(t)=q^{ON}$,
$t\in\left[  \tau_{k},\tau_{k+1}\right)  $. Hence, evaluating $h^{OFF}\left(
\tau_{k}^{-}\right)  $ from (\ref{eq: h_OFF}) over the appropriate time
interval results in%
\[%
\begin{array}
[c]{ll}%
h^{OFF}\left(  \tau_{k}^{-}\right)  & =x_{3}(\tau_{k-1}^{+})e^{-\left(
\tau_{k}-\tau_{k-1}\right)  /\sigma}\\
& +(\mu_{3}\sigma+x_{3,0})[1-e^{-\left(  \tau_{k}-\tau_{k-1}\right)  /\sigma
}]+\tilde{\zeta}_{3}(\tau_{k})
\end{array}
\]
and it follows directly from (\ref{eq: h_ON}) that $h^{ON}\left(  \tau_{k}%
^{+}\right)  =x_{3}(\tau_{k}^{+})$.

As in the previous case, continuity due to conservation of mass applies, so
that evaluating $\Delta_{f}^{1}(\tau_{k})\equiv f_{k}^{1}(\tau_{k}%
^{-})-f_{k+1}^{1}(\tau_{k}^{+})$ yields%
\begin{equation}%
\begin{array}
[c]{l}%
\Delta_{f}^{1}(\tau_{k},\zeta_{3}\left(  \tau_{k}\right)  )=\left\{
\alpha_{1}\left[  1+\phi_{\alpha}^{OFF}(\tau_{k}^{-})\right]  ^{-1}\right. \\
\text{ }-\alpha_{1}\left[  1+\phi_{\alpha}^{ON}(\tau_{k}^{+})\right]
^{-1}-\beta_{1}\left[  1+\phi_{\beta}^{OFF}(\tau_{k}^{-})\right]  ^{-1}\\
\text{ }+\beta_{1}\left[  1+\phi_{\beta}^{ON}(\tau_{k}^{+})\right]  ^{-1}\\
\text{ }\left.  +\frac{m_{1}}{x_{3,0}}\left[  h^{OFF}\left(  \tau_{k}%
^{-}\right)  -x_{3}(\tau_{k})\right]  \right\}  \cdot x_{1}(\tau_{k})
\end{array}
\label{eq: deltaF1 OFF/ON}%
\end{equation}

The term $\tau_{k,i}^{\prime}$ corresponds to the event time derivative with
respect to $\theta_{i}$ at event time $\tau_{k}$ and its derivation will be
detailed in Lemma 1.

Similarly for $x_{2}(t)$, we have, for $i=1,2$,%
\begin{equation}
x_{2,i}^{\prime}(\tau_{k}^{+})=x_{2,i}^{\prime}(\tau_{k}^{-})+\left[
f_{k}^{2}(\tau_{k}^{-})-f_{k+1}^{2}(\tau_{k}^{+})\right]  \cdot\tau
_{k,i}^{\prime} \label{eq: x2_prime OFF/ON}%
\end{equation}
where $\tau_{k,i}^{\prime}$ will be derived in Lemma 1. Evaluating $\Delta
_{f}^{2}(\tau_{k})\equiv f_{k}^{2}(\tau_{k}^{-})-f_{k+1}^{2}(\tau_{k}^{+})$
from (\ref{eq: x2 dynamics}), and making the appropriate simplifications due
to continuity, we obtain%
\begin{equation}%
\begin{array}
[c]{ll}%
\Delta_{f}^{2}(\tau_{k},\zeta_{3}\left(  \tau_{k}\right)  ) & =\frac
{\alpha_{2}d}{x_{3,0}}\left[  x_{3}(\tau_{k})-h^{OFF}\left(  \tau_{k}%
^{-}\right)  \right]  \cdot x_{2}(\tau_{k})\\
& -\frac{m_{1}}{x_{3,0}}\left[  h^{OFF}\left(  \tau_{k}^{-}\right)
-x_{3}(\tau_{k})\right]  \cdot x_{1}(\tau_{k})
\end{array}
\label{eq: deltaF2 OFF/ON}%
\end{equation}

Finally, for $x_{3}(t)$, (\ref{eq: Boundary cond}) can be easily seen to yield
$x_{3,i}^{\prime}(\tau_{k}^{+})=x_{3,i}^{\prime}(\tau_{k}^{-})+\frac{x_{3,0}%
}{\sigma}\cdot\tau_{k,i}^{\prime}$, $i=1,2$.

In the case of the "clock" state variable, $z_{1}(t)$ is continuous in $t$ at
$t=\tau_{k}$, while $z_{2}(t)$ is discontinuous. Based on
(\ref{eq: Boundary cond}) and (\ref{eq: z1 dynamics}), it is simple to verify
that, for $i=1,2$,%
\[
z_{1,i}^{\prime}(\tau_{k}^{+})=z_{1,i}^{\prime}(\tau_{k}^{-})-\tau
_{k,i}^{\prime}%
\]
Applying (\ref{eq: Reset fn.}) to the reset function defined in
(\ref{eq: z2 dynamics}) yields $z_{2,i}^{\prime}(\tau_{k}^{+})=0$, $i=1,2$,
i.e., the value of $z_{2,i}^{\prime}(t)$ is reset to zero whenever event
$e_{2}$ takes place.

Note that, since $z_{j,i}^{\prime}(t)=z_{j,i}^{\prime}(\tau_{k}^{+})$,
$t\in\left[  \tau_{k},\tau_{k+1}\right)  $, we will have that $z_{j,i}%
^{\prime}(\tau_{k}^{-})=z_{j,i}^{\prime}(\tau_{k-1}^{+})$, $j,i=1,2$.
Moreover, the sample path of our SHA consists of a sequence of alternating
$e_{1}$ and $e_{2}$ events, which implies that $z_{1,i}^{\prime}(\tau_{k}%
^{-})=0$ if event $e_{1}$ occurred at $\tau_{k-1}$, while $z_{2,i}^{\prime
}(\tau_{k}^{-})=0$ if event $e_{2}$ occurred at $\tau_{k-1}$. As a result,%
\begin{equation}
z_{1,i}^{\prime}(\tau_{k}^{+})=\left\{
\begin{array}
[c]{ll}%
-\tau_{k,i}^{\prime} & \text{if event }e_{2}\text{ occurs at }\tau_{k}\\
0 & \text{otherwise}%
\end{array}
\right.  \label{eq: State deriv. z1}%
\end{equation}
and%
\begin{equation}
z_{2,i}^{\prime}(\tau_{k}^{+})=\left\{
\begin{array}
[c]{ll}%
-\tau_{k,i}^{\prime} & \text{if event }e_{1}\text{ occurs at }\tau_{k}\\
0 & \text{otherwise}%
\end{array}
\right.  \label{eq: State deriv. z2}%
\end{equation}

We now proceed with a general result which applies to all events defined for
our SHA model. Let us denote the time of occurrence of the $j$th state
transition by $\tau_{j}$, and define its derivative with respect to the
control parameters as $\tau_{j,i}^{\prime}\equiv\frac{\partial\tau_{j}%
}{\partial\theta_{i}}$, $i=1,2$. We also define $f_{j}^{x_{n}}\left(  \tau
_{j}\right)  \equiv\dot{x}_{n}(\tau_{j})$, $n=1,\ldots,3$ and note that at
each state transition at time $\tau_{j}$ an event $e_{p}$, $p=1,2$, will take place.

\begin{lemma}
When an event $e_{p}$, $p=1,2$, occurs, the derivative $\tau_{j,i}^{\prime}$,
$i=1,2$, of state transition times $\tau_{j}$, $j=1,2,\ldots$ with respect to
the control parameters $\theta_{i}$, $i=1,2$, satisfies:%
\begin{equation}
\tau_{j,i}^{\prime}=\frac{\mathbf{1}\left[  p=i\right]  -x_{1}^{\prime}%
(\tau_{j}^{-})-x_{2}^{\prime}(\tau_{j}^{-})}{f_{j-1}^{x_{1}}(\tau_{j}%
^{-})+f_{j-1}^{x_{2}}(\tau_{j}^{-})} \label{eq: Lemma 1}%
\end{equation}
where $\mathbf{1}\left[  p=i\right]  $ is the usual indicator function.
\end{lemma}

\begin{proof}
We begin with an occurrence of event $e_{1}$ which causes a transition from
state $q^{ON}$ to state $q^{OFF}$ at time $\tau_{j}$. This implies that
$g_{j}(x,\theta)=x_{1}+x_{2}-\theta_{1}=0$. As a result, $\frac{\partial
g_{k}}{\partial x_{1}}=$ $\frac{\partial g_{k}}{\partial x_{2}}=1$,
$\frac{\partial g_{k}}{\partial x_{3}}=\frac{\partial g_{k}}{\partial z_{i}%
}=\frac{\partial g_{k}}{\partial\theta_{2}}=0$, $i=1,2$, and $\frac{\partial
g_{k}}{\partial\theta_{1}}=-1$, and it is simple to verify that
(\ref{eq: Lemma 1}) follows from (\ref{eq: Event time deriv}).

Next, consider event $e_{2}$ at time $\tau_{j}$, leading to a transition from
state $q^{OFF}$ to state $q^{ON}$. In this case, $g_{j}(x,\theta)=x_{1}%
+x_{2}-\theta_{2}=0$, so that $\frac{\partial g_{k}}{\partial x_{1}}=$
$\frac{\partial g_{k}}{\partial x_{2}}=1$, $\frac{\partial g_{k}}{\partial
x_{3}}=\frac{\partial g_{k}}{\partial z_{i}}=\frac{\partial g_{k}}%
{\partial\theta_{1}}=0$, $i=1,2$, and $\frac{\partial g_{k}}{\partial
\theta_{2}}=-1$. Substituting into (\ref{eq: Event time deriv}) once again
yields (\ref{eq: Lemma 1}).
\end{proof}

We note that the numerator in (\ref{eq: Lemma 1}) is determined using
(\ref{eq: Deriv x1 t1}) and (\ref{eq: Deriv x2 t1}) if $q(\tau_{j}^{-}%
)=q^{ON}$, or (\ref{eq: Deriv x1 t3minus}) and (\ref{eq: Deriv x2 tkminus}) if
$q(\tau_{j}^{-})=q^{OFF}$. Moreover, the denominator in (\ref{eq: Lemma 1}) is
computed using (\ref{eq: x1 dynamics})-(\ref{eq: x2 dynamics}) and it is
simple to verify that, if event $e_{1}$ takes place at time $\tau_{j}$,%
\[%
\begin{array}
[c]{l}%
f_{j-1}^{x_{1}}(\tau_{j}^{-})+f_{j-1}^{x_{2}}(\tau_{j}^{-})=\alpha_{1}\left[
1+\phi_{\alpha}^{ON}(\tau_{j}^{-})\right]  ^{-1}\cdot x_{1}(\tau_{j})\\
\text{ \ }-\left\{  \beta_{1}\left[  1+\phi_{\beta}^{ON}(\tau_{j}^{-})\right]
^{-1}+\lambda_{1}\right\}  \cdot x_{1}(\tau_{j})+\mu_{1}\\
\text{ \ }+\left[  \alpha_{2}\left(  1-d\frac{h^{ON}\left(  \tau_{j}%
^{-}\right)  }{x_{3,0}}\right)  -\beta_{2}\right]  \cdot x_{2}(\tau_{j})\\
\text{ \ \ }+\zeta_{1}(\tau_{j})+\zeta_{2}(\tau_{j})
\end{array}
\]
and, if event $e_{2}$ takes place at time $\tau_{j}$,%
\[%
\begin{array}
[c]{l}%
f_{j-1}^{x_{1}}(\tau_{j}^{-})+f_{j-1}^{x_{2}}(\tau_{j}^{-})=\alpha_{1}\left[
1+\phi_{\alpha}^{OFF}(\tau_{j}^{-})\right]  ^{-1}\cdot x_{1}(\tau_{j})\\
\text{ \ \ }-\left\{  \beta_{1}\left[  1+\phi_{\beta}^{OFF}(\tau_{j}%
^{-})\right]  ^{-1}+\lambda_{1}\right\}  \cdot x_{1}(\tau_{j})+\mu_{1}\\
\text{ \ }+\left[  \alpha_{2}\left(  1-d\frac{h^{OFF}\left(  \tau_{j}%
^{-}\right)  }{x_{3,0}}\right)  -\beta_{2}\right]  \cdot x_{2}(\tau_{j})\\
\text{ \ \ }+\zeta_{1}(\tau_{j})+\zeta_{2}(\tau_{j})
\end{array}
\]

This completes the derivation of all state and event time derivatives required
to apply IPA to the hybrid automaton model of prostate cancer. In what
follows, we shall derive the cost derivatives corresponding to the performance
metric defined in (\ref{eq: Sample fn}).

\subsection{Cost Derivative}

Let us denote the total number of on and off-treatment periods (complete or
incomplete) in $\left[  0,T\right]  $ by $K_{T}$. Also let $\xi_{k}$ denote
the start of the $k^{th}$ period and $\eta_{k}$ denote the end of the $k^{th}$
period (of either type). Finally, let $M_{T}=\lfloor\frac{K_{T}}{2}\rfloor$ be
the total number of complete on-treatment periods, and $\Delta_{m}^{ON}$
denote the duration of the $m^{th}$ complete on-treatment period, where
clearly%
\[
\Delta_{m}^{ON}\equiv\eta_{m}-\xi_{m}%
\]

\begin{theorem}
The derivative of the sample function $L(\theta)$ with respect to the control
parameters satisfies:%
\begin{equation}%
\begin{array}
[c]{l}%
\frac{dL(\theta)}{d\theta_{i}}=\frac{W_{1}}{T}\overset{K_{T}}{\underset
{k=1}{\sum}}\int_{\xi_{k}}^{\eta_{k}}\left[  \frac{x_{1,i}^{\prime}%
(\theta,t)+x_{2,i}^{\prime}(\theta,t)}{PSA_{init}}\right]  dt\\
\text{ \ \ }+\frac{W_{2}}{T}\overset{M_{T}}{\underset{m=1}{\sum}}\Delta
_{m}^{ON}\cdot\left(  \eta_{m,i}^{\prime}-\xi_{m,i}^{\prime}\right) \\
\text{ \ \ }-\frac{W_{2}}{T}\mathbf{1}\left[  \lfloor\frac{K_{T}}{2}%
\rfloor\neq\frac{K_{T}}{2}\right]  \cdot\xi_{M_{T}+1,i}^{\prime}\cdot\left(
T-\xi_{M_{T}+1}\right)
\end{array}
\label{eq: Theorem 1}%
\end{equation}

\end{theorem}

\begin{proof}
We assume, without loss of generality, that the start of our sample path will
coincide with the start of the first on-treatment period. Note also that we
choose to end our sample path at time $T$, and that this choice is independent
of $\theta_{i}$, $i=1,2$. Consequently, we will have $\left[  0,T\right]
\equiv\left[  \xi_{1},\eta_{K_{T}}\right]  $, which implies that
$\frac{\partial\xi_{1}}{\partial\theta_{i}}=\frac{\partial\eta_{K_{T}}%
}{\partial\theta_{i}}=0$, $i=1,2$. Recall, from the definition of an
intermittent hormone therapy scheme, that the sample path of our SHA will
consist of alternating on and off-treatment periods. Since $z_{1}(t)=0$ when
$q(t)=q^{OFF}$, we can rewrite (\ref{eq: Sample fn}) as%
\begin{equation}%
\begin{array}
[c]{l}%
L\left(  \theta,x(0),z(0),T\right)  =\frac{W_{1}}{T}\overset{K_{T}}%
{\underset{k=1}{\sum}}\int_{\xi_{k}}^{\eta_{k}}\left[  \frac{x_{1}%
(\theta,t)+x_{2}(\theta,t)}{PSA_{init}}\right]  dt\\
\text{ \ \ \ \ \ \ \ \ \ }+\frac{W_{2}}{T}\left[  \overset{M_{T}}%
{\underset{m=1}{\sum}}\int_{\xi_{k}}^{\eta_{k}}z_{1}(t)dt+\int_{\xi_{M_{T}+1}%
}^{T}z_{1}(t)dt\right]
\end{array}
\label{eq: Sample cost}%
\end{equation}

Note that our sample path can either \textit{(a)} end with an incomplete
on-treatment period, or \textit{(b)} end with an incomplete off-treatment
period. In (\ref{eq: Sample cost}), we assume that \textit{(a)} holds, since
\textit{(b)} is a special case of \textit{(a)} for which $\int_{0}^{T}%
z_{1}(t)dt=\overset{M_{T}}{\underset{m=1}{\sum}}\int_{\xi_{m}}^{\eta_{m}}%
z_{1}(t)dt$. Observe that the end of an on-treatment period is coupled with
the start of the subsequent off-treatment period, i.e., $x_{i}\left(  \eta
_{k}\right)  =x_{i}\left(  \xi_{k+1}\right)  $, $i=1,2$, $k=1,\ldots,K_{T}-1$.
Using this notation and taking the derivative of (\ref{eq: Sample cost})
yields%
\begin{equation}%
\begin{array}
[c]{l}%
\frac{dL(\theta)}{d\theta_{i}}=\frac{W_{1}}{T\cdot PSA_{init}}\overset
{K_{T}-1}{\underset{k=1}{\sum}}\int_{\xi_{k}}^{\xi_{k+1}}\left[
x_{1,i}^{\prime}(\theta,t)+x_{2,i}^{\prime}(\theta,t)\right]  dt\\
+\frac{W_{1}}{T\cdot PSA_{init}}\overset{K_{T}-1}{\underset{k=1}{\sum}}\left[
x_{1}\left(  \xi_{k+1}\right)  +x_{2}\left(  \xi_{k+1}\right)  \right]
\frac{\partial\xi_{k+1}}{\partial\theta_{i}}\\
-\frac{W_{1}}{T\cdot PSA_{init}}\overset{K_{T}-1}{\underset{k=1}{\sum}}\left[
x_{1}\left(  \xi_{k}\right)  +x_{2}\left(  \xi_{k}\right)  \right]
\frac{\partial\xi_{k}}{\partial\theta_{i}}\\
+\frac{W_{1}}{T\cdot PSA_{init}}\int_{\xi_{K_{T}}}^{T}\left[  x_{1,i}^{\prime
}(\theta,t)+x_{2,i}^{\prime}(\theta,t)\right]  dt\\
+\frac{W_{1}}{T\cdot PSA_{init}}\left[  x_{1}(T)+x_{2}(T)\right]
\frac{\partial T}{\partial\theta_{i}}\\
-\frac{W_{1}}{T\cdot PSA_{init}}\left[  x_{1}(\xi_{K_{T}})+x_{2}(\xi_{K_{T}%
})\right]  \frac{\partial\xi_{K_{T}}}{\partial\theta_{i}}\\
+\frac{W_{2}}{T}\overset{M_{T}}{\underset{m=1}{\sum}}\left[  \int_{\xi_{m}%
}^{\eta_{m}}z_{1,i}^{\prime}(t)dt+z_{1}(\eta_{m}^{-})\frac{\partial\eta_{m}%
}{\partial\theta_{i}}-z_{1}(\xi_{m}^{+})\frac{\partial\xi_{m}}{\partial
\theta_{i}}\right]  \\
+\frac{W_{2}}{T}\int_{\xi_{M_{T}+1}}^{T}z_{1,i}^{\prime}(t)dt+z_{1}%
(T^{-})\frac{\partial T}{\partial\theta_{i}}-z_{1}(\xi_{M_{T}+1}^{+}%
)\frac{\partial\xi_{M+1}}{\partial\theta_{i}}%
\end{array}
\label{eq: Sample deriv}%
\end{equation}

Observe that the first two summation terms in (\ref{eq: Sample deriv})
simplify to%
\begin{equation}%
\begin{array}
[c]{c}%
\overset{K_{T}-1}{\underset{k=1}{\sum}}\left[  x_{1}\left(  \xi_{k+1}\right)
+x_{2}\left(  \xi_{k+1}\right)  \right]  \frac{\partial\xi_{k+1}}%
{\partial\theta_{i}}\\
-\overset{K_{T}-1}{\underset{k=1}{\sum}}\left[  x_{1}\left(  \xi_{k}\right)
+x_{2}\left(  \xi_{k}\right)  \right]  \frac{\partial\xi_{k}}{\partial
\theta_{i}}\\
=\left[  x_{1}(\xi_{1})+x_{2}(\xi_{1})\right]  \frac{\partial\xi_{1}}%
{\partial\theta_{i}}\\
+\left[  x_{1}(\xi_{K_{T}})+x_{2}(\xi_{K_{T}})\right]  \frac{\partial
\xi_{K_{T}}}{\partial\theta_{i}}%
\end{array}
\label{eq: Simplification}%
\end{equation}
Further note that the sixth term in (\ref{eq: Sample deriv}) cancels out with
the second term on the right hand side of (\ref{eq: Simplification}).
Moreover, it is clear from (\ref{eq: z1 dynamics}) that $z_{1}(\xi_{M_{T}%
+1}^{+})=z_{1}(\xi_{m}^{+})=0$ and $z_{1}(\eta_{m}^{-})=\eta_{m}-\xi_{m}$,
$m=1,\ldots,M_{T}$. Since $z_{j,i}^{\prime}(t)=z_{j,i}^{\prime}(\tau_{k}^{+}%
)$, $j,i=1,2$, over any interevent interval $\left[  \tau_{k},\tau
_{k+1}\right)  $, and recalling that $\frac{\partial T}{\partial\theta_{i}%
}=\frac{\partial\xi_{1}}{\partial\theta_{i}}=0$, the last two terms in
(\ref{eq: Sample deriv}) simplify to%
\begin{align*}
&  \frac{W_{2}}{T}\overset{M_{T}}{\underset{m=1}{\sum}}\left[  z_{1,i}%
^{\prime}(\xi_{m}^{+})\left(  \eta_{m}-\xi_{m}\right)  +\left(  \eta_{m}%
-\xi_{m}\right)  \frac{\partial\eta_{m}}{\partial\theta_{i}}\right]  \\
&  +\frac{W_{2}}{T}z_{1,i}^{\prime}(\xi_{M_{T}+1}^{+})\left(  T-\xi_{M_{T}%
+1}\right)
\end{align*}
Recall that $\xi_{m}$ is the start of the $m$th on-treatment period, which
necessarily corresponds to the $m-1$th occurrence of event $e_{2}$. Hence, it
follows from (\ref{eq: State deriv. z1}) that $z_{1,i}^{\prime}(\xi_{m}%
^{+})=-\xi_{m,i}^{^{\prime}}$, $m=1,\ldots,M_{T+1}$. As a result,
(\ref{eq: Sample deriv}) can be further simplified to%
\begin{equation}%
\begin{array}
[c]{ll}%
\frac{dL(\theta)}{d\theta_{i}} & =\frac{W_{1}}{T\cdot PSA_{init}}%
\overset{K_{T}-1}{\underset{k=1}{\sum}}\int_{\xi_{k}}^{\xi_{k+1}}\left[
x_{1,i}^{\prime}(\theta,t)+x_{2,i}^{\prime}(\theta,t)\right]  dt\\
& +\frac{W_{1}}{T\cdot PSA_{init}}\int_{\xi_{K_{T}}}^{T}\left[  x_{1,i}%
^{\prime}(\theta,t)+x_{2,i}^{\prime}(\theta,t)\right]  dt\\
& +\frac{W_{2}}{T}\left[  \overset{M_{T}}{\underset{m=1}{\sum}}-\xi
_{m,i}^{^{\prime}}\left(  \eta_{m}-\xi_{m}\right)  +\left(  \eta_{m}-\xi
_{m}\right)  \eta_{m,i}^{^{\prime}}\right]  \\
& -\frac{W_{2}}{T}\xi_{M_{T}+1}^{^{\prime}}\left(  T-\xi_{M_{T}+1}\right)
\end{array}
\label{eq: Sample deriv full}%
\end{equation}
The result in (\ref{eq: Sample deriv full}) is obtained under the assumption
that our sample path ends with an incomplete on-treatment period. This
condition is satisfied when $\lfloor\frac{K_{T}}{2}\rfloor\neq\frac{K_{T}}{2}%
$. If this is not the case, i.e., if the sample path ends with an incomplete
off-treatment period and $\lfloor\frac{K_{T}}{2}\rfloor=\frac{K_{T}}{2}$, the
last term in (\ref{eq: Sample deriv full}) can be disregarded. It is then
straightforward to verify that (\ref{eq: Sample deriv full}) can be rewritten
as (\ref{eq: Theorem 1}).
\end{proof}

Observe that evaluating (\ref{eq: Theorem 1}) requires knowledge of
$x_{1,i}^{\prime}(\theta,t)$ and $x_{2,i}^{\prime}(\theta,t)$ over all on and
off-treatment periods. Over on-treatment periods, this can be determined using
(\ref{eq: x1_prime ON}) and (\ref{eq: x2_prime ON}), which ultimately depend
on (\ref{eq: h_ON}), so that it is necessary to evaluate the integral of the
noise process $\zeta_{3}(t)$. In the case of off-treatment periods,
(\ref{eq: x1_prime OFF}) and (\ref{eq: x2_prime OFF}) must be used, so that
(\ref{eq: h_OFF}) must be evaluated, for which knowledge of the integral of
the noise process $\zeta_{3}(t)$ is also needed. In the second and third terms
in (\ref{eq: Theorem 1}), $\Delta_{m}^{ON}$ can be computed using timers whose
start and end times are observable events, while $\eta_{m,i}^{\prime}$ and
$\xi_{m,i}^{\prime}$, $m=1,\ldots,M_{T}$, and eventually $\xi_{M_{T}%
+1,i}^{\prime}$, can be computed through (\ref{eq: Lemma 1}), which requires
knowledge of the noise processes $\zeta_{1}(t)$ and $\zeta_{2}(t)$ evaluated
at event times only.

\addtolength{\textheight}{-3cm}


\section{CONCLUSION}

Biological systems are inherently sensitive to physiologic cues, such as the
timing and dosage of drugs and related procedures. Hence, performing
sensitivity analysis of the mathematical models that infer patient response to
such cues will aid the development of personalized treatment schemes. Such
sensitivity analysis should not only allow for evaluating the sensitivies with
respect to model parameters, but also, and most importantly, provide
sensitivity estimates with respect to controllable parameters in a therapy.
The methodology we have laid out in this paper addresses both these needs.

Indeed, this work is the first step towards the development of a methodology
for personalized therapy design applicable to stochastic models of cancer
progression. We illustrate our analysis with a case study of optimal
IAS\ therapy design for prostate cancer. For such, we propose an SHA model to
describe the evolution of prostate cancer under IAS therapy and derive a cost
metric in terms of the desired outcome of IAS\ therapy that is parameterized
by an appropriately chosen controllable parameter vector. The problem of
optimal personalized therapy design is then formulated as the search for the
parameter values which minimize our cost metric. In this context, we apply
Infinitesimal Perturbation Analysis (IPA) and derive unbiased gradient
estimates of the cost metric with respect to the controllable vector of
interest, which can be used for sensitivity analysis of therapy schemes.

More importantly, however, since (\ref{eq: Theorem 1}) provides an
\emph{unbiased} estimate of $dJ(\theta)/d\theta_{i}$, it can also be
ultimately used for therapy estimation and optimization. To this end, it is
possible to implement an algorithm for updating the value of $dL(\theta
)/d\theta_{i}$ after each observed event. Such value can then be used to
compute an optimal $\mathbf{\theta}^{\ast}$ through an interative optimization
procedure of the form $\theta_{i,l+1}=\theta_{i,l}-\rho_{l}H_{i,l}\left(
\mathbf{\theta}_{l},x(0),T,\omega_{l}\right)  $, where $\rho_{l}$ is the step
size at the $l$th iteration, $l=0,1,\ldots$, and $\omega_{l}$ denotes a sample
path from which data are extracted and used to compute $H_{i,l}\left(
\mathbf{\theta}_{l},x(0),T,\omega_{l}\right)  $ defined to be an estimate of
$dJ(\theta)/d\theta_{i}$. The sample paths can be generated through simulation
of existing models (e.g., our SHA model), so that, by varying the model
parameters, different patient behaviors can be analyzed. Alternatively, it is
possible to apply our IPA estimators to real patient data obtained from
clinical trials (available e.g., in \cite{Bruchovsky2006} and
\cite{Bruchovsky2007}), and ultimately contrast the optimal therapy scheme
$\mathbf{\theta}^{\ast}$ with the prescribed one. Our ongoing work involves
implementing the IPA estimators derived in this work for personalized
IAS\ therapy design.



\bibliography{CHAbib}

\end{document}